\documentclass[a4paper]{amsart}
\usepackage[latin1]{inputenc} 
\usepackage{amsmath,amssymb,amsthm,color,enumerate,url,dsfont,fancyhdr,mathbbol,enumerate}
\usepackage[pdftex]{hyperref}
\usepackage{tikz}
\usepackage{color}
\usepackage{amsmath} 

\newtheorem{theo}{Theorem}
\newtheorem{lemma}[theo]{Lemma}
\newtheorem{prop}[theo]{Proposition}

\newtheorem{rem}[theo]{Remark}

\newtheorem{conj}{Conjecture}

\newtheorem{ex}[theo]{Example}
 \def\mG{\mathsf{G}}
 \def\mV{\mathsf{V}}
 \def\mE{\mathsf{E}}

 \def\mv{\mathsf{v}}
 \def\me{\mathsf{e}}
 \def\mw{\mathsf{w}}
 
 \DeclareMathOperator{\Real}{Re}

\providecommand{\abs}[1]{\lvert#1\rvert}

\providecommand{\Norm}[1]{\left\lVert#1\right\rVert}
\providecommand{\norm}[1]{\left\lVert#1\right\rVert}

\providecommand{\Ker}{{\mbox{Ker}}}
\providecommand{\sgn}[1]{\mbox{sgn}}
\providecommand{\ran}[1]{\mbox{ran}}
\providecommand{\deg}[1]{\mbox{deg}}
\providecommand{\diag}[1]{\mbox{diag}}

\providecommand{\Real}{\mbox{Re}}
\providecommand{\det}[1]{\mbox{det}}

\providecommand{\C}{\mathds{C}}
\providecommand{\Z}[1]{\mathds{Z}}
\providecommand{\N}[1]{\mathds{N}}
\providecommand{\K}[1]{\mathds{K}}
\providecommand{\In}[1]{\mbox{I}} 
\providecommand{\Ex}[1]{\mbox{E}} 
\providecommand{\dd}[1]{\mbox{d}}
\providecommand{\ddx}[1]{\mbox{dx}}
\providecommand{\ddy}[1]{\mbox{dy}}
\DeclareMathOperator{\au}{\underline{a}}

\DeclareMathOperator{\Id}{Id}

\numberwithin{equation}{section}
\numberwithin{theo}{section}

\author{Amru Hussein}
\address{Amru Hussein, FB 08 - Institut f\"{u}r Mathematik,
Johannes Gutenberg-Universit\"{a}t Mainz, Staudinger Weg 9,
55099 Mainz,
Germany}
\email{hussein@mathematik.uni-mainz.de}
\author{Delio Mugnolo}
 \address{Delio Mugnolo, Institut f\"ur Analysis, Universit\"at Ulm, 89069 Ulm, Germany}
\email{delio.mugnolo@uni-ulm.de}

\thanks{This article was partially completed during a visit of the first author at the University of Ulm under the financial support of the Land Baden--W\"urttemberg in the framework of the \emph{Juniorprofessorenprogramm} -- research project on `Symmetry methods in quantum graphs'. Parts of this article are contained in the first author's PhD thesis. Both authors would like to express their gratitude to Stefano Cardanobile and Stefan Rotter (BCN Freiburg) for pleasant and interesting discussions on neuronal modeling.} 

\title[Quantum graphs with mixed dynamics]{Quantum graphs with mixed dynamics:\\ the transport/diffusion case}

\begin{document}
\date{\today}
\begin{abstract}
We introduce a class of partial differential equations on metric graphs associated with mixed evolution: on some edges we consider diffusion processes, on other ones transport phenomena. This yields a system of equations with possibly nonlocal couplings at the boundary. We provide sufficient conditions for these to be governed by a contractive semigroup on a Hilbert space naturally associated with the system. We show that our setting is also adequate to discuss specific systems of diffusion equations with boundary delays.
\end{abstract} 
  \maketitle

\section{Introduction}
In the literature, usually considered problems concern networks whose ongoing dynamical processes are homogeneous: on \textit{each} link, the evolution is modeled as a transport, a diffusion, a wave, a beam, etc. However, many physical models consist of coexisting, interacting processes of different type. On different edges, a different kind of dynamics may take place; or else, one may introduce fictitious, auxiliary edges in the model in order to describe certain phenomena (like delays) in a more efficient way. Accordingly, our aim is to discuss a Cauchy problem for a system of partial differential equations of mixed type: a former part of the system will satisfy  heat equations, whereas a latter part will satisfies transport equations. We assume the interactions of the different subsystems to take place only at the boundary. In this way we can translate our system into a vector-valued abstract Cauchy problem with suitable, non-standard coupled boundary conditions. 
The topic of partial differential equations on networks has become very popular in the last fifteen years, mostly due to its connections with quantum chaos which has motivated the introduction of the keyword `quantum graphs'.
Partial differential equations on metric graphs that are associated with \emph{observables} are usually studied in the framework of the theory of quantum graphs. This looks natural, considering the underlying physical motivations. So far, the large majority of investigations have been devoted to self-adjoint Laplacians and other second-order elliptic operators. We mention the comprehensive surveys~\cite{KuchBer,Kuc04} and the article \cite{KS1999}. On the other hand, also some ten years ago some pioneering investigations on systems of transport equations on metric graphs have been commenced in~\cite{KraSik05}. The study of the first derivative on metric graphs inside the mathematical physics community seems to be less common and dwells on the interpretation of $i\frac{d}{dx}$ as the momentum operator. We are only aware of the pioneering study~\cite{Car99} and of the later articles~\cite{Exn12,JorPedTia11}, where some of the results of~\cite{Dor05,KraSik05} have been re-discovered and complemented by thorough spectral investigations.

The aim of this paper is to connect these two theories, which have so far been studied separately.  The easiest case of the coupling of \emph{one} diffusion and \emph{one} transport equation (both possibly vector-valued) has been considered by Gastaldi and Quarteroni in~\cite{GasQua89}.  This work was presented by its authors as a first step towards the coupling of a Euler and a Navier-Stokes equation. A related viscosity analysis has been performed in~\cite{BarMor05}. 

In~\cite{GasQua89}, well posedness of the problem has been studied by means of viscosity methods under certain (relatively strict) coercivity assumptions. We are going to weaken the coercivity assumptions and study a general system consisting of finitely many coupled intervals of either type. We are going to do so by considering a large class of coupled boundary conditions for operator matrices including both second and first derivatives. 

While the proposed dynamics actually seems to reflect the observed phenomena, it is very difficult to make an educated guess when it comes to propose \emph{natural} transmission conditions. It seems that the search for correct transmission conditions, which is easy in the case of purely diffusive \cite{Lum80} or pure transport-like systems \cite{KraSik05}, is much less trivial in the mixed case. As Gastaldi and Quarteroni put it in~\cite[section 1]{GasQua89}, `When coupling Euler and Navier-Stokes equations, the proper interface conditions are not obvious, in advance. A possible way of deducing them is to see the coupled problem as a limit of two coupled Navier-Stokes problems with vanishing viscous terms (in either one of the regions)'. But even this limiting process is quite delicate, as the analysis in~\cite{BarMor05,GasQua89} shows.

Our main result, theorem~\ref{mth}, states that systems of the kind discussed above are well posed under a large class of transmission conditions. This may hopefully suggest suitable conditions in specific settings. Furthermore, we complement a spectral analysis of the considered system of typical features of semigroups for heat and transport equations. In particular, contractivity of a semigroup guarantees that the norm of solutions to the initial value problem are not larger than the norm of the initial data. 

Unlike in~\cite{GasQua89}, a possible motivation for the introduction of our setting arises from some biomathematical considerations. More precisely, it is known that electric signal coming from a neuron undergoes a certain synaptic delay before reaching another neuron, and cannot turn back: this suggests to model this process by a system of diffusion (in the dendrites or axons) and transport (in the synapse) equations.  This is explained in some detail in section~\ref{subs:dendro}.

In section~\ref{sec:opermetr}, we introduce the specific class of boundary conditions we are going to investigate. In section~\ref{sec:main}, we discuss, in dependence of the boundary conditions, some sufficient conditions for well posedness of the problem in the $L^2$-setting,
collecting all proofs in section~\ref{proofs}. In section~\ref{sec:spectr}, we derive a secular equation for our system and an explicit representation of the resolvent.

It turns out that our setting can be adapted to discuss certain classes of coupled systems of diffusion equations with boundary delays. We describe in Subsection~\ref{subs:dendro-2} how to apply our abstract results to the  motivating problem from Section~\ref{subs:dendro}.

Most of the present paper has been written with a main focus on evolution equations of parabolic type. Even when the evolution equation is not governed by an analytic semigroup, we still think of our problem as a sort of degenerate diffusion equation. For this reason, in this paper we investigate some properties that are typical for diffusive problems -- and less so for quantum mechanical ones. In order to keep this article as self-contained as possible, we recall in the appendix all the definition and abstract notions  we are going to need. We refer to~\cite{EngNag06} for more details and proofs.

\section{Delayed diffusion equations via mixed PDEs on graphs}\label{subs:dendro}

Let us discuss a concrete example, which involves the problems outlined in the introduction and suggests a possible interplay with delayed diffusion equations. It is a classical idea, thoroughly developed e.g.\ in~\cite{BatPia05}, that delays can be mathematically modeled by means of (auxiliary) transport phenomena. While this theory is classical whenever the delay is `distributed' -- i.e. it acts on each point of the domain of diffusion -- it is less standard and technically more involved whenever the delays only concern the boundary values. The situation we want to describe is the following. We consider a ramified structure with an ongoing diffusion process. Incoming particles are absorbed in some nodes and stored there for some time (which depends solely on the features of the node itself), before triggering a flow in the incident edges. This phenomenon can be described by attaching to each such node a fictitious loop, which the particles have to cross before reaching the adjacent edges. A comparable approach has been presented in the recent article~\cite{BayDorRha12}.

As a tentative motivation of the investigation of this class of problems, we briefly discuss a (much simplified) model of a \emph{dendrodendritic chemical synapse}, referring the reader to any introductory textbook on neuronal modeling (e.g.,~\cite[chapters 2 and 9]{Sco02} and~\cite[chapter 5]{DayAbb05}) for basic notions and some miscellaneous models, including delayed ones. Also for the sake of simplicity, we drop the absorption term and therefore replace the cable equation usually discussed in theoretical neuroscience by a simpler diffusion equation. (This is in fact a bounded perturbation; hence it can be neglected when it comes to discussing well posedness). We also remark that our model essentially applies to the case of less exotic axodendritic synapses, if we linearize the active transport phenomena typical of the latter. Of course, although the situation considered here is motivated by a toy neuronal model, it can also be regarded as a simple example of a general diffusion equation with boundary delays.

Consider two dendrites (modeled by two intervals $\me_1,\me_2$) that are incident in the synapse $\mv$, which is terminal endpoint of $\me_1$ and initial endpoint of $\me_2$. The synaptic input coming from $\me_1$ undergoes a delay of $\tau_{\rm del}=1$ before reaching $\me_2$ and cannot turn back. For the sake of simplicity, we also impose \emph{sealed end} conditions on the first dendrite $\me_1$ as well as on the second endpoint of $\me_2$: 

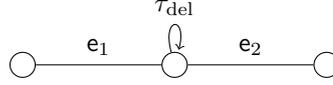
\begin{center}
\begin{tikzpicture}
\node(pseudo) at (-1,0){};
\node(0) at (0,0)[shape=circle,draw] {};
\node(1) at (2,0)[shape=circle,draw] {};
\node(2) at (4,0)[shape=circle,draw] {};
\path [-]
  (0)      edge                 node [above]  {$\me_1$}     (1)
  (1)      edge                 node [above]  {$\me_2$}     (2);
  \path [->]
  (1)      edge [loop above]    node [above]  {$\tau_{\rm del}$}     ();
\end{tikzpicture}
\end{center}
Of course, longer chains of neurons might be modeled in a similar way.

The synaptic input is an action potential that lets neurotransmitters be released by synaptic vesicles. Experimental observations seem to suggest that no obvious (linear) algebraic relation exists between the pre- and post-synaptic potential in the dendrites -- i.e. between the boundary values of the unknowns $u_1$ and $u_2$ in the diffusion equations. We propose to discuss this system by a network diffusion problem with boundary delay
\begin{equation}\tag{BD}
\left\{\begin{array}{rcll}
\dot{u}_1(t,x)&=& u''_1(t,x), &t\geq 0,\; x\in(0,1),\\
\dot{u}_2(t,x)&=& u''_2(t,x), &t\geq 0,\; x\in(0,1),\\
u_1(t,1)&=&u'_1(t,1),&t\geq 0,\\
-u'_2(t,0)&=& u'_1(t-1,1), &t\geq 0,\\
u'_1(t,0)&=&0, &t\geq 0,\\
u'_2(t,1)&=&0, &t\geq 0,\\
u_1(0,x)&=&f_1(0,x),&x\in(0,1),\\
u_2(0,x)&=&f_2(0,x),&x\in(0,1),\\
u_1(t,1)&=& {f}_{\rm del}(t),&t\in [-1,0].\\
\end{array}
\right.
\end{equation}

The first boundary condition, of Robin-type, can be interpreted by saying that part of the ions reaching the presynaptic nerve terminal is reflected into the entrance dendrite. The second condition is actually the relevant one. It reflects the fact that the conduction speed of the signal in the dendrites is approximatively constant, at least within the same cortical area. In other words, even if the postsynaptic potential will in general have a different amplitude from the presynaptic one, their speed of propagation will be the same. Finally, it is easy to convince ourselves that the above problem is undetermined if the last initial condition on the delay term is not imposed.

In order to implement the delay feature into the node conditions, we introduce an edge $\me_{\rm del}$ and an unknown $u_{\rm del}$ to model the transport of neurotransmitters in the synaptic cleft between the pre- and the postsynaptic neurons, hence effectively introducing a delay phenomenon:

\begin{center}
\begin{tikzpicture}
\node(pseudo) at (-1,0){};
\node(0) at (0,0)[shape=circle,draw] {};
\node(1) at (2,0)[shape=circle,draw] {};
\node(2) at (4,0)[shape=circle,draw] {};
\node(3) at (6,0)[shape=circle,draw] {};
\draw [->] (2.2,0) -- (3,0);
\path [-]
  (0)      edge                 node [above]  {$\me_1$}     (1)
  (1)      edge                 node [above]  {$\me_{\rm del}$}     (2)
  (2)      edge                 node [above]  {$\me_2$}     (3);
\end{tikzpicture}
\end{center}
Our aim is now to replace the fourth, i.e. the delayed equation in ${\rm(BD)}$ by two node conditions in the endpoints of $\me_{\rm del}$. More precisely, we impose that 
\begin{equation}\label{cond1}
u'_1(t,1)=u_ {\rm del}(t,0), \qquad t\geq 0.
\end{equation}
as well as
\begin{equation}\label{cond3}
-u'_2(t,0)=u_{\rm del}(t,1),\qquad t\geq 0.
\end{equation}
In other words, the flow of postsynaptic potential and the flow of presynaptic potential agree, even if its transmission undergoes a certain delay (which we have normalized).

We are hence led to consider an (undelayed) initial boundary value problem 
\begin{equation}\tag{BD$'$}
\left\{\begin{array}{rcll}
\dot{u}_1(t,x)&=& u''_1(t,x), &t\geq 0,\; x\in(0,1),\\
\dot{u}_{\rm del}(t,x)&=& -u'_{\rm del}(t,x), &t\geq 0,\; x\in(0,1),\\
\dot{u}_2(t,x)&=& u''_2(t,x), &t\geq 0,\; x\in(0,1),\\
u'_1(t,1)&=&u_ {\rm del}(t,0),& t\geq 0,\\
u'_1(t,1)&=& u_1(t,1), &t\geq 0,\\
-u'_2(t,0)&=& u_{\rm del}(t,1),&t\geq 0,\\
u'_1(t,0)&=&0, &t\geq 0,\\
u'_2(t,1)&=&0, &t\geq 0,\\
u_1(0,x)&=&f_1(0,x),&x\in(0,1),\\
u_2(0,x)&=&f_2(0,x),&x\in(0,1),\\
u_{\rm del}(0,x)&=& {f}_{\rm del}(0,x),&x\in (0,1),\\
\end{array}
\right.
\end{equation}
i.e. we have got rid of the boundary delay by passing to the a larger state space. Observe that the above model is intrinsically non-symmetric in the sense that potential can only flow from dendrite $\me_1$ to $\me_2$, but not vice versa. This is a typical feature of \emph{chemical synapses}, as opposed to \emph{electric} ones. 

One can check that the problems ${\rm (BD)}$ and ${\rm (BD')}$ are equivalent. It should be observed that transport-based synaptic transmission models are not very common in the literature. Due to their numerical and analytic complexity they are actually often replaced by convective-diffusive (or even purely diffusive) models. A convincing pleading of a transport approach can be found in~\cite[sections~2 and 6]{Whe98}. In the following, a general framework is discussed that allows us to deal not only with this motivating example.

\section{Boundary conditions for a mixed operator on a metric graph}\label{sec:opermetr}

We use the standard construction for the singular manifold on which the dynamic of the system is going to take place: we consider finitely many intervals of finite length that are connected to realize a simple finite metric graph $\mG$ with node set $\mV$ and edge set $\mE$. For the purposes of this paper, it is important to consider a partition of $\mE$ in two disjoint subsets $\mE_d$ and $\mE_t$, which are going to represent the edges on which \emph{diffusion} and \emph{transport} phenomena are going to take place, respectively. Hence, we denote by $\me_{d1},\ldots,\me_{dD}$ and $\me_{t1},\ldots,\me_{tT}$ the edges of the metric graph. To each edge $\me_{di}$ or $\me_{tj}$, we associate a length $a_{di}$ or $a_{tj}$, respectively, which in turn determines an orientation of each edge from $0$ to the other endpoint. More specifically, this orientation leads to the representation 
\[
\me=\overrightarrow{(\mv,\mw)},\qquad \hbox{for }\mv,\mw \in \mV,
\]
and in this case we write $\me(0)=\mv$, $\me(a)=\mw$. (Here and in the following, $a$ denotes the generic length -- that is, either $a_{di}$ or $a_{tj}$, depending on the context.)

The structure of the network is given by the $|\mV|\times |\mE|$-\emph{outgoing} and \emph{ingoing incidence matrices}
${\mathcal I}^+:=(\iota^+_{\mv \me})$ and ${\mathcal I}^-:=(\iota^-_{\mv \me})$ defined by
\begin{equation}\label{inout}
\iota^+_{\mv \me}:=\left\{
\begin{array}{rl}
1, & \hbox{if } \me(0)=\mv,\\
0, & \hbox{otherwise},
\end{array}
\right.
\qquad\hbox{and}\qquad
\iota^-_{\mv \me}:=\left\{
\begin{array}{rl}
1, & \hbox{if } \me(a)=\mv,\\
0, & \hbox{otherwise.}
\end{array}
\right.
\end{equation}
The matrix ${\mathcal I}:=(\iota_{\mv \me})$ defined by ${\mathcal I}:={\mathcal I}^+-{\mathcal I}^-$ is the incidence matrix of $\mG$. Furthermore, let $\Gamma(\mv)$ be the set of all edges incident in $\mv$, i.e. 
$$\Gamma(\mv):=\left\{\me \in \mE: \me(0)=\mv\hbox{ or } \me(a)=\mv\right\}.$$ 
For the sake of notational simplicity, if $\me=\overrightarrow{(\mv,\mw)}$, we denote the value of a function $u_\me:[0,a]\to\mathbb C$ on $\me$ at $0$ and $a$ by $u_\me(\mv)$ and $u_\me(\mw)$, respectively. With an abuse of notation, we also set
$u_\me(\mv)=0$ whenever $\me\notin\Gamma(\mv)$. 
By assumption, $\mE$ is a finite measure space. The space $L^2(\mE)$ of square integrable functions defined on the intervals associated with the edges in $\mE$ becomes a Hilbert space with respect to the natural scalar product 
\[
\langle u,v\rangle:=\int_\mE u \, \overline{v}:=\sum_{i=1}^D \int_{0}^{a_{di}} u_{\me_{di}}(s) \overline{v_{\me_{di}}(s)} ds+ \sum_{j=1}^T \int_{0}^{a_{tj}} u_{\me_{tj}}(s) \overline{v_{\me_{tj}}(s)} ds.
\]
By $L^2(\mE_d)$ and $L^2(\mE_t)$, we denote the $L^2(\mE)$-functions with support on the edges belonging to $\mE_d$ and $\mE_t$, respectively. Note that any element $u\in L^2(\mE)$ admits the representation
\begin{equation}
\label{udutdecomp}
u\equiv \begin{bmatrix} u_d \\ u_t\end{bmatrix}, \qquad \mbox{where }u_d\in L^2(\mE_d)\hbox{ and }u_t\in L^2(\mE_t).  
\end{equation}
Finally, by $H^m(\mE_d)$ and $H^m(\mE_t)$ we denote the Sobolev spaces of square integrable functions that are $m$-times weakly differentiable on each edge with weak derivatives in $L^2(\mE_d)$ and $L^2(\mE_t)$, respectively. We stress that we are not imposing boundary conditions of any type on the functions in $H^m(\mE_d)$ or $H^m(\mE_t)$. Furthermore, $H^m_0(\mE_d)$ denote the set of all functions $u\in H^m(\mE_d)$  with 
\begin{eqnarray*}
u_{\me}^{(j)}(\mv)=u_{\me}^{(j)}(\mw)=0 & \mbox{for any edge $\me=\overrightarrow{(\mv,\mw)}\in \mE_d$ and $0\leq j\leq m-1$}.
\end{eqnarray*}
We can define $H^m_0(\mE_t)$ likewise. 

On $L^2(\mE)$, we formally consider the diagonal operator matrix
\begin{eqnarray}\label{differA}
A:=\begin{bmatrix}
\frac{d^2}{dx^2} & 0 \\ 0 & -\frac{d}{dx} 
\end{bmatrix},\quad\mbox{i.e.} & Au:=A\begin{bmatrix}
u_d \\ u_t
\end{bmatrix}:=\begin{bmatrix}
u''_d \\ -u'_t
\end{bmatrix}.
\end{eqnarray}
We are going to present a class of $m$-dissipative realizations of $A$ whose domains contain the minimal one $H_0^2(\mE_d)\oplus H^1_0(\mE_t)$ and are contained in the maximal one $H^2(\mE_d)\oplus H^1(\mE_t)$: the fact that these are the extreme relevant cases follows from the observation that in both cases the transport and the diffusion part are clearly edge-wise decoupled.

\begin{lemma}\label{lem:sch}
The vector space 
$H^1(\mE):=H^1(\mE_d)\oplus H^1(\mE_t)$ is a Hilbert space with respect to the natural inner product
\begin{equation*}
\langle u, v\rangle_{H^1(\mE)}:=\langle u', v'\rangle_{L^2(\mE_d)}+\langle u',v'\rangle_{L^2(\mE_t)}+\langle u, v\rangle_{L^2(\mE)}.
\end{equation*}
Also the maximal domain $D(A)$ of $A$, defined by 
$D(A):=H^2(\mE_d)\oplus H^1(\mE_t)$, is a Hilbert space with respect to the inner product
\[
\langle u'',v''\rangle_{L^2(\mE_d)}+\langle u, v\rangle_{H^1(\mE)}.
\]
Their embedding into $L^2(\mE)$ is dense and of $p$th Schatten class for all $p>1$. Furthermore, there exists $C>0$ such that the Gagliardo--Nirenberg-type estimate
\begin{equation}\label{gagliardo}
\|u\|^2_{L^\infty(\mE)} \leq C \|u\|_{L^{2}(\mE)} \|u\|_{H^{1}(\mE)}\qquad \hbox{for all }u\in H^1(\mE)
\end{equation}
holds.
\end{lemma}

\begin{proof}
Throughout this paper, we are assuming our graph to be compact, hence the embeddings of the Hilbert spaces $H^1(\mE_d)\oplus H^1(\mE_t)$ and  $H^2(\mE_d)\oplus H^1(\mE_t)$ into $L^2(\mE_d)\oplus L^2(\mE_t)$ are of $p$th Schatten class for all $p>1$, cf~\cite{Gra68}. 

A constant $C>0$ can always be found so that~\eqref{gagliardo} holds, as this inequality can be reduced to the case of intervals. If namely $u\in H^1(0,1)$; then
\begin{equation}\label{traceineq}
|u(y)|^2\leq 2\sqrt{2} \|u\|_{L^{2}(0,1)} \|u\|_{H^{1}(0,1)}\qquad \hbox{for all }y\in [0,1].
\end{equation}
Indeed if $u\in H^1(0,1)$ is such that $u(0)=0$, then we can write
$$
|u(y)|^2\le \int_0^1 (u^2)'(x)\, dx\qquad \hbox{for all }y\in [0,1],
$$
and by Cauchy-Schwarz's inequality, we get
$$
|u(y)|^2\leq 2 \|u\|_{L^{2}(0,1)} \|u'\|_{L^{2}(0,1)}\qquad \hbox{for all }y\in [0,1].
$$
For a general $u\in H^1(0,a)$, it suffices to apply the previous estimate to $u(x)=(1-\frac{x}{a})u(\frac{x}{a})+\frac{x}{a}u(\frac{x}{a})$.
To derive~\eqref{gagliardo} on the whole graph, it suffices to sum up the estimates obtained on {each edge separately}.\end{proof}

We now impose conditions that force the numerical range to lie in a left half-plain of the complex plain.
 An elementary integration by parts yields 
\begin{equation}\label{intbp}
\begin{array}{rcl}
\displaystyle{\langle Au,v \rangle} &=& \quad \displaystyle{ \int_{\mE_d} \langle u_d'', v_d\rangle - \int_{\mE_t}\langle u_t',v_t\rangle }\\
&=&\displaystyle{-\int_{\mE_d} \langle u'_d, v'_d\rangle + [ u_d',v_d ]_{\partial \mE_d} + \int_{E_t}\langle u_t, v_t' \rangle-[ u_t,v_t]_{\partial \mE_t},}
\end{array}
\end{equation}
for $u,v\in H^2(\mE_d)\oplus H^1(\mE_t)$, where we have introduced a notation based on the (not sign definite) sesquilinear forms
\begin{eqnarray*}
[u_d,v_d]_{\partial \mE_d}&:=&
\sum_{i=1}^{D}\left( u_{di}(a_{di})\overline{v_{di}}(a_{di})-u_{di}(0)\overline{v_{di}}(0)\right), \\
\left[u_t,v_t\right]_{\partial \mE_t}&:=&
\sum_{j=1}^{T}\left( u_{ti}(a_{tj})\overline{v_{tj}}(a_{tj})-u_{tj}(0)\overline{v_{tj}}(0)\right).
\end{eqnarray*}
While the equation~\eqref{intbp} is not particularly appealing, the real part of the associated form is	
\begin{equation}\label{forma}
\Real \langle Au,u\rangle= -\int_{\mE_d} \abs{u_d'}^2 + \Real [ u_d',u_d]_{\partial \mE_d} - \frac{1}{2}\left[u_t,u_t\right]_{\partial \mE_t},\qquad u\in H^2(\mE_d)\oplus  H^1(\mE_t).
\end{equation}
{For the sake of notational simplicity we introduce the $2|\mE_d|+|\mE_t|$ dimensional (``boundary'') Hilbert space 
\begin{eqnarray}\label{H}
\mathcal{H}= \mathcal{H}^+_d \oplus \mathcal{H}^-_d \oplus \mathcal{H}_t, & {\mathcal H}^{\pm}_d:={\mathbb C}^{|\mE_d|}, & {\mathcal H}_t:={\mathbb C}^{|\mE_t|}.
\end{eqnarray}}
We then define for all $u\in H^2(\mE_d)\oplus   H^1(\mE_t)$, the two vectors $\underline{u},\underline{\underline{u}}\in {\mathcal H}$, where
\begin{eqnarray}\label{Randwerte}
\underline{u}:=\begin{bmatrix} 
\{u_{di}(a_{di})\}_{1\le i\le D} \\ 
\{u_{di}(0)\}_{1\le i\le D} \\
2^{-\frac12}\{u_{tj}(a_{tj})+u_{tj}(0)\}_{1\le j\le T}
\end{bmatrix}
& {\mbox{and}} &
\underline{\underline{u}}:=\begin{bmatrix} 
\{u_{di}'(a_{di})\}_{1\le i\le D} \\ 
\{-u_{di}'(0)\}_{1\le i\le D} \\
 2^{-\frac12}{\{-u_{tj}(a_{tj})+u_{tj}(0)\}_{1\le j\le T}}
 \end{bmatrix}
\end{eqnarray}
are given with respect to the decomposition of $\mathcal{H}$ explained in~\eqref{H}. In particular, with this specific representation for the boundary values, equation~\eqref{forma} can be compactly re-written as	
\begin{equation}
{\Real}\langle Au,u\rangle= -\int_{\mE_d} \abs{u_d'}^2 {+ \Real} \langle \underline{\underline{u}},\underline{u}\rangle_{\mathcal{H}},\qquad u\in H^2(\mE_d)\oplus   H^1(\mE_t).\label{forma2}
\end{equation}
These computations motivates us to introduce a class of boundary conditions of the form
\begin{equation}\label{bcH}
P^\perp\underline{\underline{u}} + (L +P) \underline{u}=0,
\end{equation}
where $P$ is an orthogonal projector acting {in} ${\mathcal H}$, $P^\perp:= \Id-P$ {denotes the complementary orthogonal projector} and the matrix 	$L$ is an operator {in the subspace} $\Ker P{\subset \mathcal{H}}$ (whose extension by 0 to the whole of $\mathcal H $ we still denote by $L$). The boundary conditions can be equivalently written as
\begin{equation*}
P^\perp\underline{\underline{u}} + L \underline{u}=0 \quad \mbox{and} \quad  P\underline{u}=0.
\end{equation*}
We finally define the operator $A_{P,L}$ which is studied in this work as the operator $A$ equipped with the boundary conditions defined in~\eqref{bcH}, i.e.,
\begin{equation}
\label{APL}
\begin{array}{rcl}
 D(A_{P,L})&:=&\left\{ u\in H^2(\mE_d)\oplus   H^1(\mE_t) \mid P^\perp\underline{\underline{u}} + (L +P) \underline{u}=0 \right\},\\
 A_{P,L}u&:=&Au.
\end{array}
\end{equation}
Hence, inserting~\eqref{bcH} into \eqref{forma2} yields
\begin{equation*}
{\Real}\langle A_{P,L} u,u\rangle= -\int_{\mE_d} \abs{u_d'}^2 {- \Real} \langle L\underline{u},\underline{u}\rangle_{\mathcal{H}}.
\end{equation*}

\section{Main results}\label{sec:main}

We are finally in the position to state the main result of this article. The proof will be postponed to Section~\ref{proofs}.

\begin{theo}\label{mth}
Let $P$ be an orthogonal projector acting on ${\mathcal H}$ and $L$ a linear operator on ${\Ker}P$ such that $-L$ satisfies the condition
\begin{equation}\label{eq:maincond}
-\Real \langle Lx,x \rangle_{\mathcal{H}}\le \omega \left(|x_{d}^+|^2+|x_{d}^-|^2\right) \qquad \hbox{for all }x:=(x_{d}^+,x_{d}^-,x_t)\in {\mathcal H}=\mathcal{H}_d^+\oplus \mathcal{H}_d^-\oplus \mathcal{H}_t
\end{equation}
for some $\omega\ge 0$. Then the operator $A_{P,L}$ defined in \eqref{APL} is quasi-$m$-dissipative, and in fact $m$-dissipative whenever $-L$ is dissipative. Accordingly, the operator $A_{P,L}$ is the infinitesimal generator of a quasi-contractive (contractive, if $-L$ is dissipative) semigroup $(e^{tA_{P,L}})_{t\ge 0}$ on $L^2(\mE)$.
If in addition $P$ and $L$ have only real entries, then the semigroup generated by $A_{P,L}$ is real.
\end{theo}
Clearly,~\eqref{eq:maincond} is satisfied whenever $-L$ is dissipative -- but not only, as the examples in subsections~\ref{subs:dendro-2} and~\ref{subs:dendro-3} show.

The above theorem yields in particular that the initial value problem
\begin{eqnarray*}
\left\{ \begin{array}{rcll}
         \dfrac{\partial u}{\partial t}(t)&=& A_{P,L}  u(t), \quad & t\geq 0, \\
      \  u(0) &=& u_0\in L^2(\mE), \end{array} \right.  
\end{eqnarray*}
is well posed whenever $-L$ is dissipative or even satisfies \eqref{eq:maincond}, and the solution is given by 
\[
u(\cdot,t):=e^{tA_{P,L}}u_0,\qquad t\ge 0.
\]
Furthermore, the question whether the solution $u(\cdot,t)$ is a real-valued function for all $t>0$ if the initial data $u_0$ is real-valued can be answered in terms of the boundary conditions imposed at the vertices.

\begin{rem}\label{rem:decoupl}
We call the boundary conditions~\eqref{bcH} \emph{type-decoupling} if
\begin{itemize}
\item $P$ is an orthogonal projection of $\mathcal H$ onto $\{0\}$, $\mathcal H$, ${\mathbb C}^{2|\mE_d|}\times\{0\}$, or $\{0\}\times {\mathbb C}^{|\mE_t|}$, and additionally
\item $L$ is a block-diagonal matrix with respect to the decomposition of $\mathcal H$ into ${\mathbb C}^{2|\mE_d|}$ and ${\mathbb C}^{|\mE_t|}$.
\end{itemize}
In other words, the boundary conditions are type-decoupling if actually no interaction between the boundary values in $L^2(\mE_d)$ and $L^2(\mE_t)$ takes place: this is of course the most trivial case, because the dynamics of the system can be effectively reduced to that of two distinct, non-interacting systems -- a diffusive one and a transport one.

One sees that the semigroup generated by $A_{P,L}$ is not irreducible if the boundary conditions~\eqref{bcH} are type-decoupling. In particular, for type decoupling boundary conditions theorem~\ref{mth} reproduces the classical results known for the diffusion equation and the transport equation, which have so far been studied separately, as recalled in the introduction. 
\end{rem}

We conclude by stating a conjecture on the time-dependent behavior of solutions to the initial value problem considered here.

\begin{conj}
Let $\mE_t\not=\emptyset$. If a function $f$ is supported in $\me\in\mE_t$, then the semigroup generated by $A_{P,L}$ will shift its profile until the support of $e^{tA_{P,L}}f$ hits an endpoint of $\me$. Because $e^{tA_{P,L}}f$ has in this lapse of time the same profile of $f$, the semigroup cannot be immediately smoothing. In particular, it cannot be analytic -- in fact, not even immediately differentiable. On the other hand, it seems reasonable to imagine that the semigroup smoothens the profile of a function as soon as it reaches an edge in $\mE_d$. We conjecture that if $\mE_d\not=\emptyset$, then the semigroup is differentiable for all $t>T$, i.e. $(T,\infty)\ni t\mapsto T(t)x\in X$ is differentiable for all $x\in X$, where $T$ is the length of the longest path inside $\mE_t$ (to compute taking into account the possibly coupled boundary conditions).
\end{conj}

\subsection{Analysis of the motivating example}\label{subs:dendro-2}

Let us now discuss our motivating problem ${\rm (BD')}$ in the general framework we have just introduced. Observe that ${\rm (BD')}$ can be seen as an abstract Cauchy problem on $L^2(\mE)$ -- with $|\mE_d|=2$ and $|\mE_t|=1$ -- and the boundary conditions can be written as in~\eqref{bcH}, with
\begin{eqnarray*}
L:= \begin{bmatrix} 1 & 0 & 0 & 0& 0 \\
        0 & 0 & 0 & 0 & 0 \\
        0 & 0 & 0 & 0 & 0 \\
         1 & 0 & 0 & 0 & - \sqrt{2} \\
         - \sqrt{2} & 0 & 0 & 0 & 1 
\end{bmatrix} & \mbox{and} & P := 0.
\end{eqnarray*} 
Since \begin{eqnarray*}
\Real L= \frac{1}{2}(L+L^*)&=&\frac{1}{2}\begin{bmatrix} 2 & 0 & 0 & 1&  - \sqrt{2} \\
        0 & 0 & 0 & 0 & 0 \\
        0 & 0 & 0 & 0 & 0 \\
         1 & 0 & 0 & 0 & - \sqrt{2} \\
         - \sqrt{2} & 0 & 0 & - \sqrt{2} & 2 
\end{bmatrix} 
\end{eqnarray*} 
is not sign definite, $-L$ is not dissipative. However, one has 
\begin{eqnarray*}
-\langle\Real L x,x\rangle = - \abs{x_{d1}^+}^2 -\abs{x_t}^2 - \frac{1}{2}\left\langle  
\begin{bmatrix} 0 &  1\\
         1 &  0  
\end{bmatrix}  \begin{bmatrix} x_{d1}^+ \\ x_{d2}^- \end{bmatrix}, \begin{bmatrix} x_{d1}^+ \\ x_{d2}^- \end{bmatrix}\right\rangle
-
\frac{1}{2}\left\langle  
\begin{bmatrix} 0 &  0&  - \sqrt{2} \\
         0 &  0 & - \sqrt{2} \\
         - \sqrt{2} &  - \sqrt{2} & 0 
\end{bmatrix}  \begin{bmatrix} x_{d1}^+ \\ x_{d2}^- \\ x_t \end{bmatrix}, \begin{bmatrix} x_{d1}^+ \\ x_{d2}^- \\ x_t  \end{bmatrix}\right\rangle.
\end{eqnarray*}  
Decomposing the matrix as above is critic, as this allows to find norms
\begin{eqnarray*}
\omega_1 :=\Norm{ 
\frac{1}{2}\begin{bmatrix}
0 &  1 \\
1 &  0 
\end{bmatrix}} =\frac{1}{2} & \mbox{and} &
\omega_2 :=\Norm{ 
\frac{1}{2}\begin{bmatrix}
0 &  0&  - \sqrt{2} \\
0 &  0 & - \sqrt{2} \\
- \sqrt{2} &  - \sqrt{2} & 0 
\end{bmatrix}} =1,
\end{eqnarray*}
and subsequently to estimate
\begin{eqnarray*}
-\langle\Real L x,x\rangle \leq   \left(\omega_1+\omega_2-1\right)\abs{x_{d1}^+}^2 +\left(\omega_1+\omega_2\right)\abs{x_{d2}^-}^2+
(\omega_2-1)\abs{x_t}^2,
\end{eqnarray*} 
whence
\begin{eqnarray*}
-\langle\Real L x,x\rangle \leq  \frac{3}{2} \left(\abs{x_{d}^+}^2 +\abs{x_{d}^-}^2\right).
\end{eqnarray*} 
This shows that \eqref{eq:maincond} is satisfied and a direct application of theorem~\ref{mth} yields the following.
\begin{prop}
The initial-boundary value problem $\rm(BD')$ is governed by a strongly continuous semigroup $(e^{tA_{P,L}})_{t\ge 0}$ on $L^2(\mE)=L^2(0,1)\times L^2(0,1)\times L^2(0,1)$. Furthermore, this semigroup is real.
\end{prop}

\begin{rem}
Observe that the above result does not really depend on our choice to consider a delay interval $\me_{\rm del}$ of unit length: we may in fact replace $(0,1)$ by an interval of arbitrary finite length.  The same holds for the lengths of the diffusion edges. Moreover, $-L$ is not dissipative; hence theorem~\ref{mth} cannot be applied to deduce contractivity of the semigroup that governs the problem.
\end{rem}

\subsection{Another example}\label{subs:dendro-3}
Observe that dissipativity of the matrix $-L$ is sufficient but not necessary for $A_{P,L}$ to be $m$-dissipative. To give a concrete example of an $m$-dissipative operator $A_{P,L}$ where $-L$ is \emph{not} dissipative, we consider the graph consisting of one transport and one diffusion edge with certain lengths $a_{d1}$ and $a_{t1}$, 
\begin{center}
\begin{tikzpicture}
\node(pseudo) at (-1,0){};
\node(0) at (0,0)[shape=circle,draw] {};
\node(1) at (2,0)[shape=circle,draw] {};
\node(2) at (4,0)[shape=circle,draw] {};
\draw [->] (2.2,0) -- (3,0);
\path [-]
  (0)      edge                 node [above]  {$\me_{d1}$}     (1)
  (1)      edge                 node [above]  {$\me_{t1}$}     (2);
\end{tikzpicture}
\end{center}
and boundary conditions as in~\eqref{bcH} by taking
\begin{eqnarray*}
P:=\begin{bmatrix} 1 & 0 & 0 \\ 0 & 0 & 0 \\ 0 & 0 & 0  \end{bmatrix} & \mbox{and} & L_{\alpha}:=\begin{bmatrix} 0 & 0 & 0 \\ 0 & 0 & 0 \\ 0 & - \sqrt{2}\alpha & 1  \end{bmatrix} , \quad \alpha >0.
\end{eqnarray*}
This corresponds to the boundary conditions
\begin{eqnarray*}
u_{d}(a_{d1})=0, \qquad u^{\prime}_{d}(0)=0, \quad \mbox{and } \quad \alpha\, u_d(0)= u_t(0).
\end{eqnarray*}
Inserting this into \eqref{forma} delivers
\begin{eqnarray*}
\Real \langle A_{P,L_{\alpha}}u,u\rangle= -\int_{\mE_d} \abs{u_d'}^2 - \frac{1}{2}\abs{u_t(a_{t1})}^2 + \frac{\alpha^2}{2}\abs{u_{d}(0)}^2, & u\in D(A_{P,L_{\alpha}}).
\end{eqnarray*}
Note that the form defined by
\begin{eqnarray*}
-\int_{\mE_d} \abs{u_d'}^2 + \frac{\alpha^2}{2}\abs{u_{d}(0)}^2, & u \in \left\{H^1(\mE_d) \mid u_d(a)=0 \right\}  
\end{eqnarray*}
is dissipative for $\alpha>0$ small enough. Since for $\left\{H^1(\mE_d) \mid u_d(a)=0 \right\}$ a Poincar\'{e} type inequality 
$\norm{u}_{L^2(\mE_d)}\leq C \norm{u_d^{\prime}}_{L^2(\mE_d)}$ holds for a constant $C>0$, one obtains from \eqref{gagliardo}
$$\abs{u_{d}(0)}^2 \leq C^{\prime}\norm{u_d^{\prime}}^2, $$
for a constant $C^{\prime}>0$, and hence 
\begin{eqnarray*}
\Real \langle A_{P,L_{\alpha}}u,u\rangle\leq 0, & \mbox{for } \left(\frac{\alpha^2}{2}C^{\prime}-1 \right)\leq 0,
\end{eqnarray*}
i.e. $A_{P,L_{\alpha}}$ is dissipative even if the operator $-L_{\alpha}$ is not. Since $L_{\alpha}$ satisfies \eqref{eq:maincond} for some $\omega>0$ for $\alpha>0$ small enough one has by theorem~\ref{mth} that the operator $A_{P,L_{\alpha}}$ is $m$-dissipative for $\alpha>0$ small.

\subsection{A different model}

Instead of considering the transport-diffusion-type Cauchy problem associated with dissipative operators, one might think of a Schr\"odinger-type Cauchy problem involving self-adjoint operators with mixed dynamics. To this aim, consider the symmetric operator  
\begin{eqnarray*}
S^0:=\begin{bmatrix}
\frac{d^2}{dx^2} & 0 \\ 0 & i\frac{d}{dx} 
\end{bmatrix},\qquad D(S^0):= H_0^2(\mE_d)\oplus H^1_0(\mE_t).
\end{eqnarray*}
It has equal deficiency indices $(2 \abs{\mE_d}+\abs{\mE_t}, 2 \abs{\mE_d}+\abs{\mE_t})$ and its adjoint $S=(S^0)^{\ast}$ is formally the same operator with domain $H^2(\mE_d)\oplus H^1(\mE_t)$. Hence there exists self-adjoint extensions and these can be parametrized in terms of boundary conditions. One defines the appropriately modified vectors of boundary values by
\begin{eqnarray*}
\bar{u}:=\begin{bmatrix} 
\{u_{di}(a_{di})\}_{1\le i\le D} \\ 
\{u_{di}(0)\}_{1\le i\le D} \\
2^{-\frac12}\{u_{tj}(a_{tj})+u_{tj}(0)\}_{1\le j\le T}
\end{bmatrix}
&\mbox{and} & 
\bar{\bar{u}}:=\begin{bmatrix} 
\{u_{di}'(a_{di})\}_{1\le i\le D} \\ 
\{-u_{di}'(0)\}_{1\le i\le D} \\
 i \cdot 2^{-\frac12}{\{-u_{tj}(a_{tj})+u_{tj}(0)\}_{1\le j\le T}}
 \end{bmatrix}
\end{eqnarray*}
to obtain the well known Hermite symplectic form on the space of boundary values
\begin{eqnarray*}
\langle S u,v \rangle -\langle u,S v \rangle 
= \left\langle 
\begin{bmatrix} \bar{u}\\ \bar{\bar{u}} \end{bmatrix}, \begin{bmatrix} 0 & -\mathds{1}_{\mathcal H} 
\\ \mathds{1}_{\mathcal H} 
& 0 \end{bmatrix} \begin{bmatrix} \bar{v}\\ \bar{\bar{v}} \end{bmatrix} \right\rangle_{\mathcal{H}^2}, 
& & u,v\in H^2(\mE_d)\oplus H^1(\mE_t). 
\end{eqnarray*}
It is known that there is a one-to-one correspondence between the self-adjoint extensions of symmetric operators and the maximal isotropic subspaces with respect to this Hermite symplectic form, see e.g. \cite{Harmer2000}. A unique parametrization of such a subspace is given in terms of a projection $P$ and a Hermitian operator $L$ acting in ${\Ker}P$. Therefore, any self-adjoint realization $S_{P,L}$ of $S^0$ is a restriction of $S$ to a domain of the type
\begin{eqnarray*}
D(S_{P,L})=\{u\in H^2(\mE_d)\oplus H^1(\mE_t) \mid P\bar{u}=0 \ \mbox{and} \ L\bar{u}+P^\perp\bar{\bar{u}}=0  \}.
\end{eqnarray*}  

The spectral theory for the operators $S_{P,L}$ can be developed on the lines of the spectral theory for the operators $A_{P,L}$ elaborated in section~\ref{sec:spectr}. In particular, it follows like in the proof of lemma~\ref{rescomp} below that the resolvents of the self-adjoint operators $S_{P,L}$ are compact, and hence their spectrum is purely discrete. Furthermore, one can obtain an explicit expression for the resolvents like that in the forthcoming proposition~\ref{resolvent}. 

The self-adjoint operators $S_{P,L}$ can be interpreted as Hamiltonians consisting of a standard Laplacian and a less usual moment-type observable. As already mentioned, moment operators on graphs have been recently studied in~\cite{Exn12}. By Stone's theorem, for self-adjoint $S_{P,L}$ the abstract Cauchy problem
\begin{eqnarray*}
\left\{ \begin{array}{rcll}
        i \frac{\partial u}{\partial t}(t)&=& S_{P,L}  u(t), \quad & t\geq 0, \\
      \  u(0) &=& u_0\in L^2(\mE), \end{array} \right.  
\end{eqnarray*}
is governed by a unitary group.

\section{Proofs of the main results}\label{proofs}

Before proving theorem~\ref{mth}, we need two preparatory lemmata.

\begin{lemma}\label{prodiss}
Let $\omega \ge 0$ be such that~\eqref{eq:maincond} is satisfied. Then $A_{P,L}$ is quasi-dissipative: more precisely,
\begin{equation}\label{omegatilde}
\Real \langle A_{P,L} u,u\rangle \le \tilde{\omega}\|u\|_{L^2(\mE)}^2:=\left(\frac{\omega^2 C^2}{4}+1\right)\|u\|_{L^2(\mE)}^2\qquad \hbox{for all $u\in D(A_{P,L})$},
\end{equation}
where $C>0$ is a constant, depending only on the total length of the graph, such that
\begin{equation}\label{gagliardo2}
\|(u_d(a),u_d(0)) \|^2\le C\|u\|_{L^2(\mE_d)} \|u\|_{H^1(\mE_d)}\qquad \hbox{for all } u\in H^1(\mE_d).
\end{equation}
If in particular $-L$ is dissipative, then $A_{P,L}$ is dissipative.
\end{lemma}
The estimate  \eqref{gagliardo2} follows from~\eqref{gagliardo}.

\begin{proof}
Take $u\in D(A_{P,L})$. Then it follows from~\eqref{forma2} that
\begin{eqnarray*}
-\Real \langle A_{P,L}u,u\rangle = \|u'\|^2_{L^2(\mE_d)}+\Real \langle L\underline{u},\underline{u}\rangle_{\mathcal{H}}.
\end{eqnarray*}
If $-L$ is dissipative, then we just estimate this by
\begin{eqnarray*}
-\Real \langle A_{P,L}u,u\rangle\ge \|u'\|^2_{L^2(\mE_d)}\ge 0.
\end{eqnarray*}
If instead $L$ only satisfies~\eqref{eq:maincond} for $\omega>0$, then it follows from~\eqref{forma2} that
\begin{equation}\label{eq:prodiss}
\begin{array}{rcl}
-\Real \langle A_{P,L}u,u\rangle&=& \|u'\|^2_{L^2(\mE_d)}+\Real \langle L\underline{u},\underline{u}\rangle_{\mathcal{H}}\\
&\ge& \|u'\|^2_{L^2(\mE_d)}-\omega \|(u_d(a),u_d(0)) \|^2\\
&\ge & \|u'\|^2_{L^2(\mE_d)}-\omega C\|u\|_{H^1(\mE_d)}\|u\|_{L^2(\mE_d)} \\
&\ge & \|u'\|^2_{L^2(\mE_d)}-\omega C\left(\frac{\epsilon}{2}\|u'\|_{L^2(\mE_d)}^2 + \frac{\epsilon}{2}\|u\|_{L^2(\mE_d)}^2 +\frac{1}{2\epsilon}\|u\|_{L^2(\mE_d)}^2\right).
\end{array}
\end{equation}
for all $\varepsilon>0$. In particular, for $\varepsilon=\frac{2}{\omega C}$ we obtain
\begin{eqnarray*}
-\Real \langle A_{P,L}u,u\rangle&\ge& -\left(\frac{\omega^2 C^2}{4}+1\right)\|u\|_{L^2(\mE)}^2.
\end{eqnarray*}
This concludes the proof.
\end{proof}

The restriction of $A$ to $H_0^2(\mE_d)\oplus H_0^1(\mE_t)$ is denoted by $A^0$. It is straightforward to verify that the adjoint of $A^0$ in the Hilbert space $L^2(\mE)$ is the operator 
\begin{eqnarray*}
B:=\begin{bmatrix}
\frac{d^2}{dx^2} & 0 \\ 0 & \frac{d}{dx} 
\end{bmatrix}, & &
D(B):=H^2(\mE_d)\oplus H^1(\mE_t).
\end{eqnarray*}
As $A_{P,L}\subset A$ is an extension of $B^*=A^0$, the operator $A_{P,L}^{\ast}$ is a restriction of $B$. Therefore, it can be described in terms of boundary conditions {imposed on $B$}. We introduce the notation 
\begin{eqnarray*} 
\widetilde{v}=\begin{bmatrix} 
\{v_{di}(a_i)\}_{1\le i\le D} \\ \{v_{di}(0)\}_{1\le i\le D} \\ 2^{-\frac12}\{(v_{tj}(0)+v_{tj}(a_j))\}_{1\le j\le T}
\end{bmatrix} & \mbox{and} & 
\widetilde{\widetilde{v}}=\begin{bmatrix} 
\{v_{di}^\prime(a_i)\}_{1\le i\le D} \\ \{-v_{di}^\prime(0)\}_{1\le i\le D} \\ 2^{-\frac12} \{(-v_{tj}(0)+v_{tj}(a_j))\}_{1\le j\le T}
\end{bmatrix}
\end{eqnarray*}
observing that 
\begin{align*} 
\widetilde{v}= \underline{v} \ \mbox{and} \ 
\widetilde{\widetilde{v}}= J \underline{\underline{v}},\qquad \hbox{where } & J:=\begin{bmatrix} \mathds{1}_{\mathbb C^{2|\mE_d|}} & 0\\ 0 & -\mathds{1}_{\mathbb C^{|\mE_t|}} \end{bmatrix}.
\end{align*}
(The change of sign in the last component is due to the change of the direction on the transport edges). 
\begin{lemma}\label{adjoint}
The adjoint operator of $A_{P,L}$ is given by
\begin{equation*}
\begin{array}{rcl}
D(A_{P,L}^{\ast})&=&\{(v_d, v_t) \in H^2(\mE_d)\oplus H^1(\mE_t) \mid (L^*+P)\widetilde{v}+ P^\perp \widetilde{\widetilde{v}}=0 \},\\
A_{P,L}^{\ast}u&=&Bu.
\end{array}
\end{equation*}
\end{lemma}

\begin{proof}
By definition, the adjoint of $A_{P,L}$ is the operator given by 
\begin{eqnarray*}
D(A_{P,L}^{\ast})&:=&\{v \in L^2(\mE) \mid \exists u \in L^2(\mE) \ \mbox{s.t.} \ \langle A_{P,L}w,v\rangle= \langle w,u\rangle \ \mbox{for all} \ w\in D(A_{P,L}) \},\\
A_{P,L}^{\ast}v&:	=&u.
\end{eqnarray*} 
To begin with, observe that following the computations in~\eqref{intbp} the operators $A$ and $B$ -- without boundary conditions -- satisfy
\begin{align*}
\langle Au,v \rangle - \langle u,Bv \rangle 
=& - \langle \underline{u},\widetilde{\widetilde{v}} \rangle_{{\mathcal H}} 
+ \langle \underline{\underline{u}},\widetilde{v} \rangle_{{\mathcal H}} \end{align*}
or rather
\begin{equation}\label{boundaryform}
\langle Au,v \rangle -\langle u,Bv \rangle 
= \left\langle \begin{bmatrix} \underline{u}\\ \underline{\underline{u}} \end{bmatrix}, \begin{bmatrix} 0 & -{\mathds 1}_{\mathcal H}
 \\ {\mathds 1}_{\mathcal H} 
 & 0 \end{bmatrix} \begin{bmatrix} \underline{v}\\ J\underline{\underline{v}} \end{bmatrix} \right\rangle_{{\mathcal H}^2}, \qquad u,v\in H^2(\mE_d)\oplus  H^1(\mE_t).
\end{equation}
Recall that $A_{P,L}^{\ast}$ is a restriction of the operator $B$. Hence $v\in D(A_{P,L}^{\ast})$ if and only if the boundary term in~\eqref{boundaryform} vanishes for all $u\in D(A_{P,L})$. The range of 
\begin{eqnarray*}
[\cdot]_{P,L}: D(A_{P,L})\rightarrow {\mathcal H}^2, & \left[u\right]_{P,L}= \begin{bmatrix}\underline{u} \\ \, \underline{\underline{u}}\, \end{bmatrix} 
\end{eqnarray*}
is ${\Ker}(L+P, \, P^\perp)$ and hence the boundary term vanishes for a {fixed} $v\in H^2(\mE_d)\oplus   H^1(\mE_t)$ and all $u\in D(A_{P,L})$ if and only if {
\begin{eqnarray*}
 \begin{bmatrix}\underline{v} \\ \, J\underline{\underline{v}} \, \end{bmatrix} \perp {\Ker}(-P^\perp,\, L+P), 
\end{eqnarray*}
taking into account that
\begin{eqnarray*}
\begin{bmatrix} 0 & \mathds{1}_{\mathcal H} \\ -\mathds{1}_{\mathcal H} & 0 \end{bmatrix} {\Ker}(L+P, \, P^\perp)={\Ker}(-P^\perp,\, L+P).
\end{eqnarray*}
}
Since the orthogonal complement of the space ${\Ker}(-P^\perp,\, L+P)$ is exactly ${\Ker}(L^*+P, \, P^\perp)$, one summarizes that $v\in D(A_{P,L}^{\ast})$ if and only if $v\in H^2(\mE_d)\oplus   H^1(\mE_t)$ and
$$ \begin{bmatrix}\underline{v} \\ \, J\underline{\underline{v}} \, \end{bmatrix}=\begin{bmatrix} \widetilde{v} \\ \, \widetilde{\widetilde{v}} \, \end{bmatrix}\in {\Ker}(L^*+P, \, P^\perp).$$ This completes the proof.
\end{proof}

\begin{proof}[Proof of theorem \ref{mth}]
It is known (\cite[corollary~II.3.17]{EngNag00}) that a sufficient condition for a densely defined operator to have $m$-dissipative closure is that both the operator and its adjoint are dissipative. By lemma~\ref{prodiss}, the operator $A_{P,L}$ is quasi-dissipative for any $L$ satisfying \eqref{eq:maincond} for some $\omega\geq 0$ and dissipative whenever $-L$ is dissipative. Like in lemma~\ref{prodiss} one proves that conversely the operator $A_{P,L}^{\ast}$ is quasi-dissipative for any $L$ satisfying \eqref{eq:maincond} for some $\omega\geq 0$ and dissipative whenever $-L$ is dissipative using
\begin{equation*}
{\Real}\langle Bu,u\rangle= -\int_{\mE_d} \abs{u_d'}^2 {+ \Real} \langle \widetilde{\widetilde{u}},\widetilde{u}\rangle_{\mathcal{H}},\qquad u\in H^2(\mE_d)\oplus   H^1(\mE_t).
\end{equation*}

To conclude the proof, it suffices to check that $A_{P,L}$ is actually closed. Because both the first and the second derivative without boundary conditions are closed operators, in our case it suffices to check that the boundary conditions are respected in the limit. This follows from the fact that $u\mapsto \underline{u}$ and $u\mapsto \underline{\underline{u}}$ are bounded operators from the Hilbert space $H^2(\mE_d)\oplus H^1(\mE_t)$ to ${\mathcal H}$. 

In order to prove reality of the semigroup it is sufficient to show that the resolvent maps real function to real functions. The resolvent is calculated explicitly in proposition~\ref{resolvent} in the forthcoming section~\ref{sec:spectr} as an integral operator. Observe that its kernel $r(\cdot,\cdot,i\kappa)$, $\kappa>0$ has real coefficients whenever the operator $P$ and $L$ has real entries. Therefore, the resolvents $(A_{P,L}-\kappa^2)^{-1}$ map real-valued functions to real-valued functions for any $\kappa^2>0$. Applying the generalized inverse Laplace transform described in \eqref{laplacetrans}, one concludes that the semigroup generated by $A_{P,L}$ is real whenever $P$ and $L$ are real.  
\end{proof}

\section{Spectral theory}\label{sec:spectr}

As we have mentioned in the appendix, cf~\eqref{laplacetrans}, knowing the resolvent of $A_{P,L}$ one can describe the strongly continuous semigroup generated by it by means of the inverse Laplace transform, for $A_{P,L}$  quasi-$m$-dissipative.

In the following the spectrum of the operator $A_{P,L}$ is analyzed and an explicit formula for the resolvent is derived. Taking into account lemma~\ref{lem:sch}, we promptly obtain the following.

\begin{lemma}\label{rescomp}
For all orthogonal projectors $P$ on ${\mathcal H}$ and all linear operators $L$ acting on ${\Ker}P$ satisfying \eqref{eq:maincond} the operators $A_{P,L}$ have resolvent of $p$th Schatten class for all $p>1$ (and, in particular, of Hilbert--Schmidt class). In particular, the operators $A_{P,L}$ have only pure point spectrum.
\end{lemma}

\subsection{Non-zero eigenvalues}
In order to determine the pure point spectrum of $A_{P,L}$, a natural Ansatz for finding eigenfunctions is to take $k\in \mathbb C\setminus\{0\}$ and to consider 
\begin{align*}
\phi(x,k)=\begin{cases} \alpha_{di}(k)e^{ikx} + \beta_{di}(k)e^{-ikx}, & x\in \me_{di}, \ i=1, \ldots, D, \\ \gamma_{tj}(k) e^{k^2 x}, & x\in \me_{tj},\ j=1, \ldots, T. \end{cases}
\end{align*}
The boundary conditions $(P+L)\underline{\phi(\cdot,k)}+P^\perp\underline{\underline{\phi(\cdot,k)}}=0$ are encoded in
\begin{align*}
[(P+L)X(k)+ P^\perp Y(k)]\begin{bmatrix}\alpha_d(k) \\ \beta_d(k) \\ \gamma_t(k) \end{bmatrix}=0,
\end{align*}
where $\{\alpha_d(k)\}_{i=1, \ldots, D}=\alpha_{di}(k)$, $\{\beta_d(k)\}_{i=1, \ldots, D}=\beta_{di}(k)$, $\{\gamma_t(k)\}_{i=j, \ldots, T}=\gamma_{tj}(k)$ are the sought after coefficients. The matrices 
\begin{align*}
X(k)&= \begin{bmatrix} e^{ik{\au}_d} & e^{-ik{\au}_d} & 0 \\ \mathds{1} & \mathds{1} & 0 \\ 0 & 0 & \tfrac{1}{\sqrt{2}}(\mathds{1}+e^{k^2 {\au}_t}) \end{bmatrix}, \\
Y(k)&= \begin{bmatrix} ik e^{ik{\au}_d} & -ik e^{-ik{\au}_d} & 0 \\ -ik & ik & 0 \\ 0 & 0 & \tfrac{1}{\sqrt{2}}(\mathds{1}-e^{k^2{\au}_t}) \end{bmatrix}
\end{align*}
acting in $\mathcal{H}$ are given with respect to the decomposition $\mathcal{H}=\mathcal{H}^+_d\oplus\mathcal{H}^-_d\oplus\mathcal{H}_t$. Here, the notation 
\begin{eqnarray*}
\{e^{\pm k^2{\au}_{t}}\}_{j,l=1, \ldots, T}=\delta_{jl} e^{\pm k^2{a}_{tj}} & \mbox{and} & \{e^{\pm ik{\au}_{d}}\}_{j,l=1, \ldots, D}=\delta_{jl} e^{\pm ik{a}_{dj}}
\end{eqnarray*}
is used, where $\delta_{jl}$ is the Kronecker delta. Accordingly, the following holds. 
\begin{prop}
For all orthogonal projectors $P$ {in} ${\mathcal H}$ and all linear operators $L$ acting {in} ${\Ker}P$, the number $-k^2\in\C\setminus\{0\}$ is an eigenvalue of $A_{P,L}$ if and only if the matrix
$$Z_{P,L}(k):=[(P+L)X(k)+P^\perp Y(k)],$$
has non trivial null space. The geometric multiplicity of $-k^2$ equals the dimension of ${\Ker}Z_{P,L}(k)$. 
\end{prop}
Hence, the secular equation is
$$\det Z_{P,L}(k)=0.$$
In general, it seems to be difficult to give precise statements on the distribution of the eigenvalues.

\begin{ex}\label{exrealspec}
Consider the graph consisting of one diffusion edge of length $a_{d1}$ and one transport edge of length $a_{t1}$. Let 
\begin{eqnarray*}
P=\begin{bmatrix} 1 & 0 & 0 \\ 0 & 2^{-1} & 2^{-1} \\ 0 & 2^{-1} & 2^{-1}\end{bmatrix} &\mbox{and} & L=0.
\end{eqnarray*}
Then the operator $A_{P,L}$ is $m$-dissipative and the secular equation becomes
\begin{eqnarray*}
\det Z_{P,L}(k)= \frac{i}{\sqrt{2}} \left[ \sin(k a_{d1}) \left(1- e^{k^2 a_{t1}}\right) + k \cos(k a_{d1}) \left(1+ e^{k^2 a_{t1}}\right) \right] =0.
\end{eqnarray*}
In particular, the spectrum of $A_{P,L}$ contains a sequence of real eigenvalues going to $-\infty$.
\end{ex}

\subsection{Eigenvalue zero}
For the eigenvalue zero, one uses for the eigenfunctions the Ansatz 
\begin{align*}
\phi_0(x)=\begin{cases} \alpha_{di}(0) + \beta_{di}(0)x, & x\in \me_{di}, \ i=1, \ldots, D, \\ \gamma_{tj}(0), & x\in \me_{tj},\ j=1, \ldots, T. \end{cases}
\end{align*}
The boundary conditions $(P+L)\underline{\phi(\cdot,0)}+P^\perp\underline{\underline{\phi(\cdot,0)}}=0$ are encoded in 
\begin{align*}
[(L+P)X^0+ P^\perp Y^0]\begin{bmatrix}\alpha_d(0) \\ \beta_d(0) \\ \gamma_t(0) \end{bmatrix}=0
\end{align*}
with
\begin{eqnarray*}
X^0= \begin{bmatrix} \mathds{1} & {\au}_d & 0 \\ \mathds{1} & 0 & 0 \\ 0 & 0 & \sqrt{2} \end{bmatrix} & \mbox{and} &
Y^0= \begin{bmatrix} 0 & \mathds{1} & 0 \\ 0 & -\mathds{1} & 0 \\ 0 & 0 & 0 \end{bmatrix},
\end{eqnarray*}
where $\{\au_d\}_{j,l=1, \ldots, D}=\delta_{jl} a_{dj}$. Hence we obtain a characterization for the eigenvalue zero.
\begin{prop}
For all orthogonal projectors $P$ {in} ${\mathcal H}$ and all linear operators $L$ acting {on} ${\Ker}P$, the operator $A_{P,L}$ has eigenvalue zero if and only if
\begin{eqnarray*}
\det Z^0_{P,L} = 0, & \mbox{where} & Z^0_{P,L}=(P+L)X^0+P^\perp Y^0.
\end{eqnarray*}
\end{prop}
In particular, the invertibility of the operator $A_{P,L}$ for $L$ satisfying \eqref{eq:maincond} is independent of the lengths of the transport edges. For the situation considered in subsection \ref{subs:dendro-2}, for example, one has that the operator $A_{P,L}$ is not invertible for any edge lengths.

\begin{ex}
Consider as in example ~\ref{exrealspec} the graph consisting of one diffusion edge of length $a_{d1}$ and one transport edge of length $a_{t1}$. Let 
\begin{eqnarray*}
P=\begin{bmatrix} 1 & 0 & 0 \\ 0 & 2^{-1} & 2^{-1} \\ 0 & 2^{-1} & 2^{-1}\end{bmatrix} &\mbox{and} & L_C=CP^\perp, 
\end{eqnarray*}
where $P^\perp=\mbox{Id}-P$ and $C\in\C$. Then 
$$\det Z^0_{P,L_C}= -2^{-\frac{1}{2}} - C a_{d1} 2^{\frac{1}{2}},$$
and therefore $A_{P,L_C}$ is invertible for $\Real  C\geq 0$.
\end{ex} 

\subsection{The resolvent operator}
In the following, an explicit formula for the resolvent is given in terms of the boundary conditions and the edge lengths. First we define the shorthand notation 
\begin{eqnarray*}
\displaystyle{\int_{\mG} u := \sum_{i\in \mE_d} \int_{0}^{a_{di}} u_{di}(x) dx + \sum_{j\in \mE_t} \int_{0}^{a_{tj}} u_{tj}(x) dx,} & \mbox{for} & u\in L^2(\mE).
\end{eqnarray*}

\begin{prop}\label{resolvent}
For all orthogonal projectors $P$ {in} ${\mathcal H}$, all linear operators $L$ acting on ${\Ker}P$ and for all $k\neq 0$ such that $-k^2\in \rho(A_{P,L})$, the resolvent operator $R(k)=(A_{P,L}+k^2)^{-1}$ is the operator given by
$$R(k)u(x):= \int_{\mG} r(x,\cdot,k)u(\cdot) $$
with integral 	kernel 
\begin{align*}
r(x,y,k):=\left\{r_0(x,y,k) - \Phi(x,k) \Sigma_{P,L}(k)\Psi(y,k) \right\} W(k).
\end{align*}
Here we have denoted 
$$ \Sigma_{P,L}(k):= Z_{P,L}(k)^{-1} [P^\perp R_1(k)+ (L+P) R_2(k)]$$ 
and furthermore
\begin{eqnarray*}
r_0(x,y,k)&:=& \begin{bmatrix} r_d(x,y,k) & 0 \\ 0 & r_t(x,y,k^2) \end{bmatrix},\\ 
W(k)&:=& \begin{bmatrix} \frac{i}{2k}\mathds{1}_{\C^{|\mE_d|}} & 0 \\ 0 & \mathds{1}_{\C^{|\mE_t|}} \end{bmatrix},\\
\{r_d(x,y,k)\}_{n,m}&:=&\delta_{n,m} e^{ik\abs{x-y}}, \quad n,m=1, \ldots, D,\\
\{r_t(x,y,k)\}_{j,l}&:=&\delta_{j,l} \left(\begin{cases}e^{k^2(x-y)}, & x<y \\ 0 , & x\geq y \end{cases}\right), \quad j,l=1,\ldots, T.
\end{eqnarray*} 
Finally, 
\begin{align*}
\Phi(x,k):= \begin{bmatrix} e^{ikx} & e^{-ikx} & 0 \\ 0 & 0 & e^{k^2 x} \end{bmatrix}, && \Psi(x,k):= \begin{bmatrix} e^{iky}& 0 \\ e^{-iky}& 0 \\ 0 & e^{-k^2y} \end{bmatrix},
\end{align*}
where the entries are diagonal matrices whose entries are functions with arguments from the corresponding edges and 
\begin{align*}
R_1(k):= \begin{bmatrix} ik \mathds{1}_{\C^{|\mE_d|}} & 0 & 0 \\ 0 & ik e^{ik{\au}_d} & 0 \\ 0 & 0 & \tfrac{1}{\sqrt{2}} \mathds{1}_{\C^{|\mE_t|}} \end{bmatrix}, 
&& R_2(k):	= \begin{bmatrix} \mathds{1}_{\C^{|\mE_d|}} & 0 & 0 \\ 0 & e^{ik{\au}_d} & 0 \\ 0 & 0 & \tfrac{1}{\sqrt{2}}\mathds{1}_{\C^{|\mE_t|}} \end{bmatrix}.
\end{align*}
\end{prop}

\begin{proof}
It is sufficient to prove that $r(x,y,k)$ defines the Green's function of the operator $(A_{P,L}+k^2)$. Consider the unperturbed operator 
$$R_0(k)u(x)= \int_{\mG} r_0(x,\cdot,k)W(k) u(\cdot)\qquad \hbox{for }u \in \left(\bigoplus_{i=1}^D C^{\infty}_0([0,a_{di}]) \bigoplus_{j=1}^T C^{\infty}_0([0,a_{tj}])\right).$$ 
The equation $(A+k^2) R_0(k)u =u$ is satisfied on the diffusion edges as $\tfrac{i}{2k} e^{ik\abs{x-y}}$ is the Green's function of $L_d(k)=\tfrac{d^2}{dx^2}-k^2$ on the whole real line (this follows from standard arguments using the Fourier transform). By continuing functions $u_i\in C_0^{\infty}([0,a_{di}])$ trivially to the real line the claim follows. Similarly, the diagonal entries of $r_t(x,y,k)$ are the Green's function for $L_t(k)=-\tfrac{d}{dx}-k^2$ on the whole real line, which follows from standard arguments from the theory of ordinary differential equations. Again by continuing functions $u_j$ in $C^{\infty}_0([0,a_{tj}])$ trivially to the real line the claim follows. 

Note that for the correction term one has 
$$(A+k^2) \int_{\mG} \Phi(x,k) \Sigma_{P,L}(k)\Psi(\cdot,k) W(k)u(\cdot)  =0.$$ 
Therefore $(A+k^2) R(k)u =u$. As $\left(\bigoplus_{i=1}^D C^{\infty}_0([0,a_{di}]) \bigoplus_{j=1}^T C^{\infty}_0([0,a_{tj}])\right)$ is dense in $L^2(\mE)$ and for all $k\neq 0$ such that $-k^2\in \rho(A_{P,L})$ $r(\cdot,\cdot,k)$ defines a bounded linear operator on $L^2(\mE\times \mE)$. One concludes by density that $(A+k^2) R(k)u =u$ for all $u\in L^2(\mE)$. 

It remains to prove that $R(k)u\in D(A_{P,L})$. Observe that for all $a>0$ and all $u\in L^2[(0,a)]$,
\begin{align*}
\bigg[ \int_{0}^{a} e^{ik\abs{x-y}}u(y)dy\bigg]_{x=0}=& \int_0^{a} e^{iky}u(y)dy, \\
\bigg[ \int_{0}^{a} e^{ik\abs{x-y}}u(y)dy\bigg]_{x=a}=& e^{ika} \int_0^{a} e^{-iky}u(y)dy, \\
\bigg[-\frac{d}{dx} \int_{0}^{a} e^{ik\abs{x-y}}u(y)dy \bigg]_{x=0}=& ik \int_0^{a} e^{iky}u(y)dy, \\
\bigg[\frac{d}{dx} \int_{0}^{a} e^{ik\abs{x-y}}u(y)dy \bigg]_{x=a}=& ik e^{ika} \int_0^{a} e^{-iky}u(y)dy,
\end{align*}
 and considering only the edge $[0,a]$,
\begin{align*}
 \bigg[\int_{0}^{a}r_t(x,y,k^2) u(y)dy \bigg]_{x=0} &= \int_0^{a} e^{-k^2y}f(y)dy, \\
\bigg[ \int_{0}^{a} r_t(x,y,k^2)u(y)dy \bigg]_{x=a}&= 0.
\end{align*}
This gives in the matrix notation for $u\in L^2(\mE)$ and $v\in \mathcal H$
\begin{eqnarray*}
 \underline{R_0(k)u(x)}=R_2(k)\int_{\mG}\Psi(k,\cdot)W(k)u(\cdot), & \underline{\Phi(x,k)v}=X(k)v, \\
 \underline{\underline{R_0(k)u(x)}}=R_1(k)\int_{\mG}\Psi(k,\cdot)W(k)u(\cdot), & \underline{\underline{\Phi(x,k)v}}=Y(k)v.
\end{eqnarray*}
Therefore, 
\begin{align*}
\underline{\int_{\mG} r(x,\cdot,k)u(\cdot)}=& \left(R_2(k) - X(k)\Sigma_{P,L}(k) \right)\int_{\mG}\Psi(\cdot,k)W(k)u(\cdot), \\ 
\underline{\underline{\int_{\mG} r(x,\cdot,k)u(\cdot)}}=& \left(R_1(k) - Y(k)\Sigma_{P,L}(k) \right)\int_{\mG}\Psi(\cdot,k)W(k)u(\cdot),
\end{align*}
and hence for all $u\in L^2(\mE)$,
\begin{align*}
(P+L) \underline{\underline{\int_{\mG} r(x,\cdot,k)f(\cdot)}} + P^\perp\underline{\int_{\mG} r(x,\cdot,k)f(\cdot)} =0.
\end{align*}
This proves that $r(\cdot,\cdot,k)$ is the Green's function for the mixed problem.
\end{proof}

\section{Appendix: A reminder of semigroup theory}

In this final section, we are going to recollect some  results from the general theory of strongly continuous  semigroups, in order to make the technique of this article more transparent to the reader less familiar with it. In particular, we are interested in the case of generators that are neither skew-adjoint nor self-adjoint. We refer to~\cite{AreBatHie01,EngNag00,EngNag06} for a comprehensive introduction and overview to modern semigroup theory.

Let in the following $H$ be a complex Hilbert space with scalar product $\langle \cdot, \cdot\rangle$ and induced norm $\norm{\cdot}$. It is well known that all strongly continuous semigroups $(T(t))_{t\ge 0}$ of bounded linear operators on a Hilbert space are \emph{exponentially bounded}, i.e. there exist constants $M\ge 1$ and $\omega\in \mathbb R$ such that
\[
\|T(t)\|\le Me^{\omega t}\qquad \hbox{for all }t\ge 0.
\]
Furthermore, a semigroup is called
\begin{itemize}
\item \emph{$\omega$-quasi-contractive}, for some $\omega \in \mathbb R$, if
\[
\|T(t)\|\le e^{\omega t}\qquad \hbox{for all }t\ge 0.
\]
\item \emph{quasi-contractive} if it is $\omega$-quasi-contractive for some $\omega \in \mathbb R$, and finally
\item \emph{contractive} if  it is $\omega$-quasi-contractive for $\omega=0$.
\end{itemize}
For some $\omega\in \mathbb R$, a closed and densely defined operator $A$ is called   \emph{$\omega$-quasi-dissipative} if
\[
\Real \langle Au,u\rangle \le \omega \|u\|^2\qquad \hbox{for all }u\in D(A).
\] 
For a given $\omega \in \mathbb R$, the operator $A$ is called  \emph{$\omega$-quasi-$m$-dissipative} if it is $\omega$-quasi-dissipative and additionally  $\lambda-A$ is surjective for some $\lambda>\omega$. An operator $A$ is called \emph{quasi-$m$-dissipative} (resp., \emph{quasi-dissipative}) if it is \emph{$\omega$-quasi-$m$-dissipative} (resp., \emph{$\omega$-quasi-dissipative}) for some $\omega \in \mathbb R$.

By the Lumer--Phillips Theorem, a semigroup is $\omega$-quasi-contractive if and only if its (necessarily closed and densely defined) infinitesimal generator $A$ is $\omega$-quasi-$m$-dissipative, cf~\cite[theorem~3.4.5]{AreBatHie01}.
While the above range condition can be sometimes hard to check directly, it is known that an operator is $\omega$-quasi-$m$-dissipative if in particular both the operator and its adjoint are $\omega$-quasi-dissipative.

In addition to their central role in the Lumer--Phillips Theorem, quasi-$m$-dissipative operators are important since it is possible to represent the semigroup generated by them by means of a suitable generalization of the inverse Laplace transform. More precisely, if  $A$ is the $\omega$-quasi-$m$-dissipative infinitesimal generator of the semigroup $T(t)$, then the formula
\begin{equation}\label{laplacetrans}
T(t)u= \lim_{n\to \infty}\int_{\varepsilon- i n}^{\varepsilon+ i n} e^{t\lambda} \left(A-\lambda\right)^{-1}  u \, d\lambda,\qquad u\in H,
\end{equation}
holds for any $\varepsilon>\omega$, see~\cite[theorem~3.12.2]{AreBatHie01}. 

Let us now consider the case where $H$ is a complex-valued $L^2(X)$-space for some $\sigma$-finite measure space $(X,\mu)$. Then a bounded linear operator $T$ on $H$ is called
\begin{itemize}
\item \emph{real} if $Tf(x)\in \mathbb R$ for $\mu$-a.e.\ $x\in X$, whenever $f(x)\in \mathbb R$ for  $\mu$-a.e.\ $x\in X$;
\item \emph{irreducible} if  it does not leave invariant any non-trivial ideal of $L^2(X)$ -- i.e. it does not leave invariant the subspace $L^2(\tilde{X})$ of $L^2(X)$ for any measurable subset $\tilde{X}$ of $X$ with $\tilde{X}\not=\emptyset$ and $\tilde{X}\not=X$.
\end{itemize}
An operator semigroup $(T(t))_{t\ge 0}$ on $H=L^2(X)$  is called \emph{real} (resp. \emph{irreducible}) if each operator $T(t)$ is real (resp. irreducible).

\end{document}